\documentclass[12pt]{amsart}
\usepackage{amssymb}
\usepackage{amsmath}
\usepackage{bbm}
\usepackage[dvips]{graphicx}
\usepackage{hyperref}

\numberwithin{equation}{section}

\newcommand{\be}{\begin{equation}}
\newcommand{\ee}{\end{equation}}
\newcommand{\ds}{\displaystyle}

\newcommand{\imag}{\text{\em \i}}
\newcommand{\B}{\mathcal B}
\DeclareMathOperator{\inv}{inv}
\DeclareMathOperator{\sgn}{sgn}
\DeclareMathOperator{\pf}{pf}

\newtheorem{thm}{Theorem}
\newtheorem{cor}[thm]{Corollary}
\newtheorem{lem}[thm]{Lemma}
\newtheorem{defn}{Definition}

\newtheorem{alg}{Algorithm}

\newenvironment{proofof}[1]{\medskip\noindent
   \textbf{Proof of #1:} }{\hfill $\blacksquare$\par\medskip}

\title{Determinants and Perfect Matchings}
\author{Arvind Ayyer}
\address{Arvind Ayyer\\
Department of Mathematics \\
University of California \\
Davis, CA 95616}
\email{ayyer@math.ucdavis.edu}

\date{\today}

\begin{document}

\begin{abstract}
We give a combinatorial interpretation of the determinant of a matrix
as a generating function over Brauer diagrams in two different but related ways. 
The sign of a permutation associated to its number of inversions in the Leibniz 
formula for the determinant is replaced by the number of crossings in the Brauer 
diagram.  This interpretation naturally explains why the determinant of an even
antisymmetric matrix is the square of a Pfaffian. 
\end{abstract}

\maketitle

\section{Introduction} \label{sec:intro}
There are many different formulas for evaluating the determinant of a matrix. Apart from the familiar Leibniz formula, there is Laplace formula, Dodgson's
condensation and Gaussian elimination. However, there is no formula to the best 
of our knowledge in which Cayley's celebrated formula \cite{Cayley} relating Pfaffians to determinants is transparent. In this work, we give a new
 formula which does precisely this.
 
The formula uses the notion of Brauer diagrams. These parametrize the basis elements of the so-called Brauer algebra \cite{brauer}, which is important in the representation theory of the orthogonal group. 

Brauer diagrams are perfect matchings on a certain kind of planar graph.   
We shall prove in
Theorem~\ref{thm:detm} (to be stated formally in Section~\ref{sec:doub}) that the
determinant of an $n \times n$ matrix can be expanded as a sum over
all Brauer diagrams of a certain weight function. Since perfect matchings are related to Pfaffians, we obtain a natural combinatorial interpretation of Cayley's beautiful result relating Pfaffians and determinants.
There have been some connections noted in the literature between Brauer diagrams and combinatorial objects such as Young tableaux \cite{sundaram,terada,hallew}, and Dyck paths \cite{marmat} in the past. 

The connection between determinants and perfect matchings came up
while studying the number of terms (including repetitions) in the determinants
of Hermitian matrices, which turns out to be $(2n-1)!!$. The number of
distinct terms in the determinant of symmetric and skew-symmetric matrices, on the other hand, is classical. This has been studied, among others, by
Cayley and Sylvester \cite{muir}. In particular, Sylvester showed that the number of 
distinct terms in the determinant of a skew-symmetric matrix of size $2n$
is given by $(2n-1)!! v_{n}$, where $v_{n}$ satisfies
\be
v_{n} = (2n-1)v_{n-1}-(n-1)v_{n-2}, \quad v_{0}=v_{1}=1.
\ee
Aitken \cite{Aitken} has also studied recurrences for the number of terms 
in symmetric and skew-symmetric determinants. 
The number of terms in the symmetric determinant also 
appears in a problem in the American Mathematical
Monthly proposed by Richard Stanley \cite{stanrior}.

The spirit of this work is similar to those on combinatorial interpretations of
identities and formulas in linear algebra
\cite{jackson,foata1,straub,zeil1}, combinatorial formulas for
determinants \cite{zeil2}, and for Pfaffians
\cite{halton,knuth1,mahsubvin,egecioglu}. 

The plan of the paper is as follows. Two non-standard 
representations of a matrix are given in Section~\ref{sec:faux}. 
We recall the definition of  Brauer diagrams in Section~\ref{sec:doub}.
We will also define the weight and the crossing number of a Brauer diagram, and state the main theorem  there. We will then digress to give a different combinatorial explanation for the number of terms in the determinant of these non-standard matrices in Section~\ref{sec:numterm}.  The main idea of the proof is a bijection between terms in both determinant expansions and Brauer diagrams, which will be 
given in Section~\ref{sec:bij}. We  define the crossing number for a
Brauer diagram and prove some properties about it in
Section~\ref{sec:cross}. The main result is then proved in
Section~\ref{sec:main}. 

\section{Two Different Matrix Representations} \label{sec:faux}
A word about notation: throughout, we will use $\imag$ as the
complex number $\sqrt{-1}$ and $i$ as an indexing variable.
Let $A$ be a symmetric matrix and $B$ be a skew-symmetric matrix.
Any matrix can be decomposed in two ways as a linear combination of $A$ 
and $B$, namely $A+B$ and $A+\imag B$. We denote the former by $M_{F}$ and
the latter by $M_{B}$. The terminology will be explained later.
That is, 
\be \label{defm}
(M_{F})_{i,j} = \begin{cases}
a_{i,j} +  b_{i,j} & i < j, \\
a_{j,i} -   b_{j,i} & i > j, \\
a_{i,i} & i=j,
\end{cases};
\quad
(M_{B})_{i,j} = \begin{cases}
a_{i,j} + \text{\em \i}  b_{i,j} & i < j, \\
a_{j,i} - \text{\em \i}  b_{j,i} & i > j, \\
a_{i,i} & i=j,
\end{cases}
\ee
where $a_{i,j}$ and $b_{i,j}$ are {\bf complex} indeterminates. 
For example, a generic $3\times 3$ matrix can be written in these two ways,
\be \label{mateg3}
\begin{split}
M^{(3)}_{F} &=
\begin{pmatrix}
a_{1,1} & a_{1,2}+ b_{1,2} & a_{1,3}+ b_{1,3} \\
a_{1,2}- b_{1,2} & a_{2,2} & a_{2,3}+ b_{2,3}\\
a_{1,3}- b_{1,3} & a_{2,3}- b_{2,3} & a_{3,3}
\end{pmatrix}, \\
M^{(3)}_{B} &=
\begin{pmatrix}
a_{1,1} & a_{1,2}+ \imag b_{1,2} & a_{1,3}+\imag  b_{1,3} \\
a_{1,2}- \imag b_{1,2} & a_{2,2} & a_{2,3}+ \imag b_{2,3}\\
a_{1,3}- \imag b_{1,3} & a_{2,3}- \imag b_{2,3} & a_{3,3}
\end{pmatrix}.
\end{split}
\ee
Notice that $a_{i,j}$ is defined when $i \leq j$ and $b_{i,j}$ is
defined when $i<j$. 
The determinant of the matrices is clearly a polynomial in these
indeterminates. For example, the determinant of the matrices in \eqref{mateg3} is 
given by
\be \label{deteg3}
\begin{split}
\det(M^{(3)}_{F})=& \;  a_{{1,1}}a_{{2,2}}a_{{3,3}}
-a_{{1,1}}{a_{{2,3}}}^{2} -a_{{2,2}}{a_{{1,3}}}^{2}
-a_{{3,3}}{a_{{1,2}}}^{2}\\
&+a_{{1,1}}{b_{{2,3}}}^{2}+a_{{2,2}}{b_{{1,3}}}^{2}
+a_{{3,3}}{b_{{1,2}}}^{2}
+2\,a_{{1,2}}a_{{2,3}}a_{{1,3}}\\
&-2\,a_{{1,2}}b_{{2,3}}b_{{1,3}} 
+2\,a_{{1,3}}b_{{1,2}}b_{{2,3}}-2\,a_{{2,3}}b_{{1,2}}b_{{1,3}}, \\
\det(M^{(3)}_{B})=& \; a_{{1,1}}a_{{2,2}}a_{{3,3}}
-a_{{1,1}}{a_{{2,3}}}^{2} -a_{{2,2}}{a_{{1,3}}}^{2}
-a_{{3,3}}{a_{{1,2}}}^{2}\\
&-a_{{1,1}}{b_{{2,3}}}^{2}
-a_{{2,2}}{b_{{1,3}}}^{2}-a_{{3,3}}{b_{{1,2}}}^{2}
+2\,a_{{1,2}}a_{{2,3}}a_{{1,3}}\\
&+2\,a_{{1,2}}b_{{2,3}}b_{{1,3}} 
-2\,a_{{1,3}}b_{{1,2}}b_{{2,3}}+2\,a_{{2,3}}b_{{1,2}}b_{{1,3}},
\end{split}
\ee
in these two decompositions. The number of terms in each of the formulas in 
\eqref{deteg3} is seen to be 15, which is equal to $5!!$.

\section{Brauer Diagrams} \label{sec:doub}

One of the most common representations of permutations is the {\bf two-line 
representation} or {\bf two-line diagram} of a permutation. This is also an example of a perfect matching on a complete bipartite graph.
\setlength{\unitlength}{1mm}
\begin{figure}[h!]
\centering
\begin{picture}(60, 20)
\put(14,0){1}
\put(19,0){2}
\put(24,0){3}
\put(29,0){4}
\put(34,0){5}
\put(39,0){6}
\put(44,0){7}
\put(15,4){\circle*{1}}
\put(20,4){\circle*{1}}
\put(25,4){\circle*{1}}
\put(30,4){\circle*{1}}
\put(35,4){\circle*{1}}
\put(40,4){\circle*{1}}
\put(45,4){\circle*{1}}
\put(14,16){1}
\put(19,16){2}
\put(24,16){3}
\put(29,16){4}
\put(34,16){5}
\put(39,16){6}
\put(44,16){7}
\put(15,14){\circle*{1}}
\put(20,14){\circle*{1}}
\put(25,14){\circle*{1}}
\put(30,14){\circle*{1}}
\put(35,14){\circle*{1}}
\put(40,14){\circle*{1}}
\put(45,14){\circle*{1}}
\put(15,4){\line(3,2){15}}
\put(20,4){\line(2,1){20}}
\put(25,4){\line(-1,1){10}}
\put(30,4){\line(-1,2){5}}
\put(35,4){\line(1,1){10}}
\put(40,4){\line(-2,1){20}}
\put(45,4){\line(-1,1){10}}
\end{picture}
\caption{A two-line diagram for the permutation $3641725$.}
\label{fig:permeg}
\end{figure}
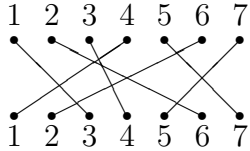

One of the advantages of a two-line diagram is  that the inversion number of
a permutation is simply the number of pairwise
intersections of the $n$ lines. In Figure~\ref{fig:permeg} above, there are 10 intersections, which is the inversion number of
the permutation $3641725$.

We will consider the complete graph on $2n$ vertices arranged in a two-line
representation.  Recall that a {\bf perfect matching} of a graph is a set of
pairwise non-adjacent edges which matches all the vertices of a
graph. The visual representations of such perfect matchings are called Brauer diagrams and are defined formally below.

\begin{defn}
Let $T$ and $B$ be the set of vertices in the top and bottom row 
respectively,  with $n$ points each, forming a two-line diagram. 
An {\bf unlabeled Brauer diagram of size $n$}, $\mu$, is a perfect matching where an edge joining two points in $T$ is called a {\bf cup}; an edge joining two points in $B$ is called a {\bf cap} and an edge joining a point in $T$ with a point in $B$ is called an {\bf arc}. For convenience, we call the former horizontal edges, and the latter,
vertical. The edges satisfy the following conditions.
\begin{enumerate}
\item Two caps may intersect in at most one point.
\item Two cups may intersect in at most one point.
\item A cap and a cup may not intersect.
\item An arc meets an arc or a cap or a cup in at most one point. 
\end{enumerate}
\end{defn}

\noindent
Let $\B_n$ be the set of unlabeled Brauer diagrams of size $n$. 
Figure~\ref{fig:pmeg} depicts an unlabeled Brauer diagram of size seven.
\setlength{\unitlength}{1mm}
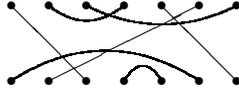
\begin{figure}[h!]
\centering
\begin{picture}(60, 20)
\put(15,4){\circle*{1}}
\put(20,4){\circle*{1}}
\put(25,4){\circle*{1}}
\put(30,4){\circle*{1}}
\put(35,4){\circle*{1}}
\put(40,4){\circle*{1}}
\put(45,4){\circle*{1}}
\put(15,14){\circle*{1}}
\put(20,14){\circle*{1}}
\put(25,14){\circle*{1}}
\put(30,14){\circle*{1}}
\put(35,14){\circle*{1}}
\put(40,14){\circle*{1}}
\put(45,14){\circle*{1}}
\put(20,4){\line(2,1){20}}
\put(25,4){\line(-1,1){10}}
\put(45,4){\line(-1,1){10}}
\qbezier(15, 4)(27.5, 12)(40, 4)
\qbezier(20, 14)(25, 10)(30, 14)
\qbezier(30, 4)(32.5, 8)(35, 4)
\qbezier(25, 14)(35, 9)(45, 14)
\end{picture}
\caption{An unlabeled Brauer diagram of size 7 with seven crossings.} \label{fig:pmeg}
\end{figure}
We now define two types of labeled Brauer diagrams.
\begin{defn}
Let $\mu \in \B_{n}$ and let $T$ be labeled with the integers 1 through $n$ from left to right. An {\bf $F$-Brauer diagram} (for forward) is a Brauer diagram where the integers 1 through $n$ are labeled left to right and an {\bf $B$-Brauer diagram} (for backward) 
is a Brauer diagram where the integers 1 through $n$ are labeled right to left.
\end{defn}
\noindent
The $F$-Brauer diagram has the same labeling as the usual two-line diagram
for a permutation.
Let $(\B_{F})_{n}$ (resp. $(\B_{B})_{n}$) be the set of $F$-Brauer diagrams
(resp. $B$-Brauer diagrams) of size $n$. Figure~\ref{fig:pmeg2} shows an example
of each type.

\setlength{\unitlength}{1mm}
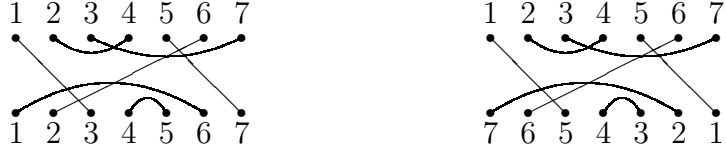
\begin{figure}[h!]
\centering
\begin{picture}(60, 20)
\put(14,0){1}
\put(19,0){2}
\put(24,0){3}
\put(29,0){4}
\put(34,0){5}
\put(39,0){6}
\put(44,0){7}
\put(15,4){\circle*{1}}
\put(20,4){\circle*{1}}
\put(25,4){\circle*{1}}
\put(30,4){\circle*{1}}
\put(35,4){\circle*{1}}
\put(40,4){\circle*{1}}
\put(45,4){\circle*{1}}
\put(14,16){1}
\put(19,16){2}
\put(24,16){3}
\put(29,16){4}
\put(34,16){5}
\put(39,16){6}
\put(44,16){7}
\put(15,14){\circle*{1}}
\put(20,14){\circle*{1}}
\put(25,14){\circle*{1}}
\put(30,14){\circle*{1}}
\put(35,14){\circle*{1}}
\put(40,14){\circle*{1}}
\put(45,14){\circle*{1}}
\put(20,4){\line(2,1){20}}
\put(25,4){\line(-1,1){10}}
\put(45,4){\line(-1,1){10}}
\qbezier(15, 4)(27.5, 12)(40, 4)
\qbezier(20, 14)(25, 10)(30, 14)
\qbezier(30, 4)(32.5, 8)(35, 4)
\qbezier(25, 14)(35, 9)(45, 14)
\end{picture}
\hfil
\begin{picture}(60, 20)
\put(14,0){7}
\put(19,0){6}
\put(24,0){5}
\put(29,0){4}
\put(34,0){3}
\put(39,0){2}
\put(44,0){1}
\put(15,4){\circle*{1}}
\put(20,4){\circle*{1}}
\put(25,4){\circle*{1}}
\put(30,4){\circle*{1}}
\put(35,4){\circle*{1}}
\put(40,4){\circle*{1}}
\put(45,4){\circle*{1}}
\put(14,16){1}
\put(19,16){2}
\put(24,16){3}
\put(29,16){4}
\put(34,16){5}
\put(39,16){6}
\put(44,16){7}
\put(15,14){\circle*{1}}
\put(20,14){\circle*{1}}
\put(25,14){\circle*{1}}
\put(30,14){\circle*{1}}
\put(35,14){\circle*{1}}
\put(40,14){\circle*{1}}
\put(45,14){\circle*{1}}
\put(20,4){\line(2,1){20}}
\put(25,4){\line(-1,1){10}}
\put(45,4){\line(-1,1){10}}
\qbezier(15, 4)(27.5, 12)(40, 4)
\qbezier(20, 14)(25, 10)(30, 14)
\qbezier(30, 4)(32.5, 8)(35, 4)
\qbezier(25, 14)(35, 9)(45, 14)
\end{picture}
\caption{The same Brauer diagram in Figure~\ref{fig:pmeg} considered as an element of $(\B_{F})_{7}$ on the left and $(\B_{B})_{7}$ on the right.} 
\label{fig:pmeg2}
\end{figure}

We draw all members of  $\B_3$ and label the matchings in Table~\ref{tab:chord3}. 

\setlength{\unitlength}{1mm}
\begin{table}[h!]
\begin{tabular}{c c c c c}
\begin{picture}(20, 20)
\put(5,5){\circle*{1}}
\put(10,5){\circle*{1}}
\put(15,5){\circle*{1}}
\put(5,10){\circle*{1}}
\put(10,10){\circle*{1}}
\put(15,10){\circle*{1}}
\put(5,5){\line(0,1){5}}
\put(10,5){\line(0,1){5}}
\put(15,5){\line(0,1){5}}
\end{picture} 
&
\begin{picture}(20, 20)
\put(5,5){\circle*{1}}
\put(10,5){\circle*{1}}
\put(15,5){\circle*{1}}
\put(5,10){\circle*{1}}
\put(10,10){\circle*{1}}
\put(15,10){\circle*{1}}
\put(5,10){\line(1,-1){5}}
\put(5,5){\line(1,1){5}}
\put(15,5){\line(0,1){5}}
\end{picture}
&
\begin{picture}(20, 20)
\put(5,5){\circle*{1}}
\put(10,5){\circle*{1}}
\put(15,5){\circle*{1}}
\put(5,10){\circle*{1}}
\put(10,10){\circle*{1}}
\put(15,10){\circle*{1}}
\put(5,10){\line(2,-1){10}}
\put(5,5){\line(1,1){5}}
\put(10,5){\line(1,1){5}}
\end{picture}
&
\begin{picture}(20, 20)
\put(5,5){\circle*{1}}
\put(10,5){\circle*{1}}
\put(15,5){\circle*{1}}
\put(5,10){\circle*{1}}
\put(10,10){\circle*{1}}
\put(15,10){\circle*{1}}
\qbezier(5, 10)(7.5, 8)(10, 10)
\put(5,5){\line(2,1){10}}
\qbezier(10,5)(12.5, 7)(15, 5)
\end{picture}
&
\begin{picture}(20, 20)
\put(5,5){\line(1,1){5}}
\put(5,5){\circle*{1}}
\put(10,5){\circle*{1}}
\put(15,5){\circle*{1}}
\put(5,10){\circle*{1}}
\put(10,10){\circle*{1}}
\put(15,10){\circle*{1}}
\qbezier(5, 10)(10, 7)(15, 10)
\put(5,5){\line(1,1){5}}
\qbezier(10,5)(12.5, 7)(15, 5)
\end{picture} \\
\begin{picture}(20, 20)
\put(5,5){\circle*{1}}
\put(10,5){\circle*{1}}
\put(15,5){\circle*{1}}
\put(5,10){\circle*{1}}
\put(10,10){\circle*{1}}
\put(15,10){\circle*{1}}
\put(5,5){\line(0,1){5}}
\put(10,5){\line(1,1){5}}
\put(15,5){\line(-1,1){5}}
\end{picture}
&
\begin{picture}(20, 20)
\put(5,5){\circle*{1}}
\put(10,5){\circle*{1}}
\put(15,5){\circle*{1}}
\put(5,10){\circle*{1}}
\put(10,10){\circle*{1}}
\put(15,10){\circle*{1}}
\put(5,10){\line(1,-1){5}}
\put(15,5){\line(-1,1){5}}
\put(5,5){\line(2,1){10}}
\end{picture}
&
\begin{picture}(20, 20)
\put(5,5){\circle*{1}}
\put(10,5){\circle*{1}}
\put(15,5){\circle*{1}}
\put(5,10){\circle*{1}}
\put(10,10){\circle*{1}}
\put(15,10){\circle*{1}}
\put(5,10){\line(2,-1){10}}
\put(10,5){\line(0,1){5}}
\put(5,5){\line(2,1){10}}
\end{picture}
&
\begin{picture}(20, 20)
\put(5,5){\circle*{1}}
\put(10,5){\circle*{1}}
\put(15,5){\circle*{1}}
\put(5,10){\circle*{1}}
\put(10,10){\circle*{1}}
\put(15,10){\circle*{1}}
\qbezier(5, 10)(7.5, 8)(10, 10)
\put(10,5){\line(1,1){5}}
\qbezier(5, 5)(10, 8)(15, 5)
\end{picture}
&
\begin{picture}(20, 20)
\put(5,5){\circle*{1}}
\put(10,5){\circle*{1}}
\put(15,5){\circle*{1}}
\put(5,10){\circle*{1}}
\put(10,10){\circle*{1}}
\put(15,10){\circle*{1}}
\qbezier(5, 10)(10, 7)(15, 10)
\put(10,5){\line(0,1){5}}
\qbezier(5, 5)(10, 8)(15, 5)
\end{picture} \\
\begin{picture}(20, 20)
\put(5,5){\circle*{1}}
\put(10,5){\circle*{1}}
\put(15,5){\circle*{1}}
\put(5,10){\circle*{1}}
\put(10,10){\circle*{1}}
\put(15,10){\circle*{1}}
\put(5,5){\line(0,1){5}}
\qbezier(10, 10)(12.5, 8)(15, 10)
\qbezier(10, 5)(12.5, 7)(15, 5)
\end{picture}
&
\begin{picture}(20, 20)
\put(5,5){\circle*{1}}
\put(10,5){\circle*{1}}
\put(15,5){\circle*{1}}
\put(5,10){\circle*{1}}
\put(10,10){\circle*{1}}
\put(15,10){\circle*{1}}
\put(5,10){\line(1,-1){5}}
\qbezier(10, 10)(12.5, 8)(15, 10)
\qbezier(5, 5)(10, 8)(15, 5)
\end{picture}
&
\begin{picture}(20, 20)
\put(5,5){\circle*{1}}
\put(10,5){\circle*{1}}
\put(15,5){\circle*{1}}
\put(5,10){\circle*{1}}
\put(10,10){\circle*{1}}
\put(15,10){\circle*{1}}
\put(5,10){\line(2,-1){10}}
\qbezier(10, 10)(12.5, 8)(15, 10)
\qbezier(5, 5)(7.5, 7)(10, 5)
\end{picture}
&
\begin{picture}(20, 20)
\put(5,5){\circle*{1}}
\put(10,5){\circle*{1}}
\put(15,5){\circle*{1}}
\put(5,10){\circle*{1}}
\put(10,10){\circle*{1}}
\put(15,10){\circle*{1}}
\qbezier(5, 10)(7.5, 8)(10, 10)
\put(15,10){\line(0,-1){5}}
\qbezier(5, 5)(7.5, 7)(10, 5)
\end{picture}
&
\begin{picture}(20, 20)
\put(5,5){\circle*{1}}
\put(10,5){\circle*{1}}
\put(15,5){\circle*{1}}
\put(5,10){\circle*{1}}
\put(10,10){\circle*{1}}
\put(15,10){\circle*{1}}
\qbezier(5, 10)(10, 7)(15, 10)
\put(10,10){\line(1,-1){5}}
\qbezier(5, 5)(7.5, 7)(10, 5)
\end{picture} \\
\end{tabular}
\vspace{0.2cm}
\caption{All Brauer diagrams belonging to $\B_3$.} \label{tab:chord3}
\end{table}

Let $\mu \in (\B_{F})_{n}$ or $(\B_{B})_{n}$. Further, let $\mu_T$ (resp. $\mu_B$) contain cups (resp. caps)  and $\mu_{TB}$ contain arcs. By convention, edges will 
be designated as ordered pairs. When the edges belong to
$\mu_T$ or $\mu_B$, they will be written in increasing order and when they
belong to $\mu_{TB}$, the vertex in the top row will be written first.
The {\bf crossing number} $\chi(\mu)$ of $\mu$ is the number of
pairwise intersections among edges in $\mu$.

\begin{table}[h!]
\begin{tabular}{|c|c|c|c|c|}
\hline
0 & 1 & 2 & 0 & 1 \\
\hline
1 & 2 & 3 & 1 & 2 \\
\hline
0 & 1 & 0 & 0 & 1 \\
\hline
\end{tabular}
\vspace{0.2cm}
\caption{Crossing numbers for all the Brauer diagrams in $\B_3$
  according to  Table~\ref{tab:chord3}.} \label{tab:cross3}
\end{table}

We now associate a weight to $\mu$, consisting of
edges $\mu_T, \mu_B$ and $\mu_{TB}$. 
 Let $a_{i,j}$ (resp. $b_{i,j}$) be unknowns defined for
$1 \leq i \leq j \leq n$ (resp. $1 \leq i < j \leq n$) and let
$(\widehat{i,j}) = (\min(i,j),\max(i,j))$. The {\bf weight of $\mu$}, $w(\mu)$,
is given by
\be
w(\mu) = \prod_{(i,j) \in \mu_T} b_{i,j} \prod_{(i,j) \in \mu_B} b_{i,j}
\prod_{(i,j) \in \mu_{TB}} a_{\widehat{i,j}}.
\ee
Note that this weight depends on
whether we consider $\mu$ as an element of $(\B_{F})_{n}$or $(\B_{B})_{n}$.
However, the formal expression is the same in both cases.
For completeness, we list the weights of all Brauer diagrams in $\B_3$ according as whether they belong in $(\B_{F})_{n}$ and $(\B_{B})_{n}$ respectively.
\begin{table}[h!]
\begin{center}
\begin{tabular}{|c|c|c|c|c|}
\hline
$a_{1,1} a_{2,2} a_{3,3} $ & $ a_{3,3} a_{1,2}^2 $ & $ a_{1,2}a_{1,3}a_{2,3} $ & 
$ a_{1,3} b_{1,2}b_{2,3} $ & $a_{1,2}b_{1,3}b_{2,3} $ \\[0.2cm]
\hline
$a_{1,1}a_{2,3}^2$ & $ a_{1,2}a_{1,3}a_{2,3}  $ & $a_{2,2} a_{1,3}^2 $ & 
$ a_{2,3}b_{1,2}b_{1,3} $ & $ a_{2,2} b_{1,3}^2  $\\[0.2cm]
\hline
$a_{1,1}b_{2,3}^2 $ & $ a_{1,2}b_{1,3}b_{2,3} $ & $a_{1,3}b_{1,2}b_{2,3}  $ & 
$ a_{3,3} b_{1,2}^2 $ & $ a_{2,3}b_{1,2}b_{1,3}$ \\[0.2cm]
\hline
\end{tabular}
\vskip 0.5cm
\begin{tabular}{|c|c|c|c|c|}
\hline 
$a_{2,2} a_{1,3}^2 $ & $ a_{1,2}a_{1,3}a_{2,3} $ & $ a_{1,1}a_{2,3}^2  $ & $ 
a_{3,3} b_{1,2}^2  $ & $ a_{2,3}b_{1,2}b_{1,3}$ \\[0.2cm]
\hline
$a_{1,2}a_{1,3}a_{2,3} $ & $ a_{3,3} a_{1,2}^2  $ & $ a_{1,1} a_{2,2} a_{3,3} $ & $ 
a_{2,3}b_{1,2}b_{1,3} $ & $ a_{2,2} b_{1,3}^2  $\\[0.2cm]
\hline
$a_{1,3}b_{1,2}b_{2,3} $ & $ a_{1,2}b_{1,3}b_{2,3} $ & $ a_{1,1}b_{2,3}^2  $ & $ 
a_{1,3} b_{1,2}b_{2,3}$ & $ a_{1,2}b_{1,3}b_{2,3}$ \\[0.2cm]
\hline
\end{tabular}
\vskip 0.2cm
\caption{Weights of all the Brauer diagrams of size $n=3$
  according to  Table~\ref{tab:chord3}. The first table describes the weights for 
  $(\B_{F})_{n}$ and the second, for $(\B_{B})_{n}$.} \label{tab:wt3}
\end{center}
\end{table}

\noindent
We are now in a position to state the main theorem.
\begin{thm} \label{thm:detm}
The determinant of an $n\times n$ matrix can be written as a sum of Brauer diagrams as,
\be
\begin{split}
\det(M_F) &= \sum_{\mu \in (\B_{F})_n} (-1)^{\chi(\mu)} w(\mu), \\
\det(M_B) &= (-1)^{\binom{n}{2}}\sum_{\mu \in (\B_{B})_n} (-1)^{\chi(\mu)}w(\mu).
\end{split}
\ee
\end{thm}

One can verify that Theorem~\ref{thm:detm} is valid for $n=3$ in both cases 
by adding all
the weights in Table~\ref{tab:wt3} times the corresponding crossing numbers in
Table~\ref{tab:cross3} for all the Brauer diagrams in
Table~\ref{tab:chord3}, and comparing with \eqref{deteg3}.

\section{The number of terms in the determinant expansion} \label{sec:numterm}
We show by a quick argument that the number of monomials
in the determinant of an $n \times n$ matrix $M_{F}$ (and for the same reason, for $M_{B}$) 
is given by $(2n-1)!!$.
This calculation is somewhat  redundant because of
Theorem~\ref{thm:detm}. The reason for this short demonstration is
that it shows why determinants should be related to perfect matchings.
To start, let $M$ be either $M_{F}$ or $M_{B}$. Recall the Leibniz formula for the determinant of $M$,
\be
\det(M) = \sum_{\pi \in S_n} (-1)^{\inv(\pi)}(M)_{1,\pi(1)} \dots
(M)_{n,\pi(n)},
\ee
where $S_n$ is the set of permutations in $n$ letters and $\inv(\pi)$ is the number
of inversion of the permutation. Usually, this
would give us $n!$ terms, of course. In the new notation,
\eqref{defm}, we obtain many more terms because each factor
$(M)_{i,\pi(i)}$ gives two terms whenever $\pi(i) \neq i$. 

To see how many terms we now have, it is best to think of permutations
according to the number and length of cycles they contain, $\pi =
C_1\dots C_k$. If a cycle $C$ is of length 1, $C=(i)$, then it
corresponds to a diagonal element $a_{i,i}$, which contributes one
term. If, on the other hand, $C$ contains $j$ entries, then there are
$j$ off diagonal elements, which give $2^j$ terms, counting
multiplicities, exactly half of which contain an odd number of
$b_{i,j}$'s. These terms will be cancelled by the permutation $\pi'$
which has all other cycles the same, and $C$ replaced by $C'$, the
reverse of $C$. Therefore, if $C$ contains $j$ entries, we effectively
get a contribution of $2^{j-1}$ terms.

The number of terms can be written as a sum
over permutations with $k$ disjoint cycles. When there are $k$
 cycles, we get $2^{n-k}$ terms. Since the number of
permutations with $k$ disjoint cycles is the unsigned Stirling number
of the first kind, $s(n,k)$, the total number of terms is
given by
\be
\sum_{k=1}^n s(n,k) 2^{n-k}.
\ee
Since  the generating function of the unsigned 
Stirling numbers of
the first kind are given by the Pochhammer symbol or rising factorial,
\be
\sum_{k=1}^n s(n,k) x^k = (x)^{(n)} \equiv x(x+1)\cdots(x+n-1),
\ee
we can calculate the more general sum,
\be
\sum_{k=1}^n s(n,k) x^{n-k} = (1+x)(1+2x)\cdots(1+(n-1)x).
\ee
Substituting $x=2$ in the above equation gives $(2n-1)!!$, the desired answer.

\section{Bijection between terms and Labeled Brauer diagrams} \label{sec:bij}
We now describe the bijection between labeled Brauer diagrams on the one hand
and permutations leading to a product of $a_{i,j}$'s and $b_{i,j}$'s
on the other. The algorithm is independent of whether we consider $\B_{F}$ or 
$\B_{B}$. Let $\mu$ be a labeled Brauer diagram.
We first state the algorithm constructing the latter
from the former.

\begin{alg} \label{alg:pmtoperm}
We start with the three sets of matchings $\mu_T, \mu_B$ and $\mu_{TB}$.
\begin{enumerate}
\item For each term $(i,j)$ in $\mu_T$ and $\mu_B$, write the term
  $b_{i,j}$ and for $(i,j)$ in $\mu_{TB}$, write the term
  $a_{\widehat{i,j}}$. 
\item Start with $\pi=\emptyset$.
\item Find the smallest integer $i_1 \in T$ not yet in $\pi$ and
  find its partner $i_2$.  That is, either $(i_1,i_2) \in \mu_{TB}$ or
  $\widehat{(i_1,i_2)} \in \mu_T$. If $i_2=i_1$, then append the cycle $(i_1)$
  to $\pi$ and repeat Step~3. Otherwise move on to Step~4.
\item If $i_{k}$ is in $T$ (resp. $B$), look for the partner of the
  other $i_{k}$ in $B$ (resp. $T$) and call it $i_{k+1}$. Note that $i_{k+1}$ 
  can be in $T$ or  $B$ in both cases.
\item Repeat Step~4 for $k$ from 2 until $m$ such that
  $i_{m+1}=i_1$. Append the cycle $(i_1,i_2,\dots,i_m$) to $\pi$.
\item Repeat Steps~3-5 until $\pi$ is a permutation on $n$ letters in
  cycle notation.
\end{enumerate}
\end{alg}

Therefore, we obtained the desired product in Step~1 and the
permutation at the end of Step~6. Here is a simple consequence of the
algorithm.

\begin{lem}
By the construction of Algorithm~\ref{alg:pmtoperm}, if the triplet
$(\mu_T, \mu_B, \\ \mu_{TB})$ leads to $\pi$, then $(\mu_B, \mu_T,\mu_{TB})$ leads to
$\pi^{-1}$.
\end{lem}
\begin{proof}
Each cycle $(i_1,i_2,\dots,i_m)$ constructed according to
Algorithm~\ref{alg:pmtoperm} by the triplet $(\mu_T, \mu_B,\mu_{TB})$ will
be constructed as $(i_1,i_m,\dots, i_2)$ by the triplet $(\mu_B,
\mu_T,\mu_{TB})$. Since each cycle will be reversed, this is the inverse
of the original permutation.
\end{proof}

We now describe the reverse algorithm.
\begin{alg} \label{alg:permtopm}
We start with a product of $a_{i,j}$'s and $b_{i,j}$'s, and a
permutation $\pi=C_1\dots C_m$ written in cycle notation such that $1
\in C_1$, the smallest integer in $\pi \setminus C_1$ belongs to
$C_2$, and so on.
\begin{enumerate}
\item For each $b_{i,j}$, we obtain a term $\widehat{(i,j)}$ which
  belongs either to $\mu_T$ or $\mu_B$ and for each $a_{i,j}$, we obtain
  one of $(i,j)$ or $(j,i)$ which belongs to $\mu_{TB}$.
\item Start with $\mu_T=\mu_B=\mu_{TB}=\emptyset$. Set $k=1$.
\item Find the first entry $i_1$ in $C_k$ and look for either
  $a_{i_1,i_2}$ or $b_{i_1,i_2}$. If the former, assign $i_2$ to $B$
  and append $(i_1,i_2)$ to $\mu_{TB}$ and otherwise, assign $i_2$ to
  $T$ and append $(i_1,i_2)$ to $\mu_T$. Set $l=2$.
\item Find either $a_{i_l,i_{l+1}}$ or $b_{i_l,i_{l+1}}$. Assign
  $i_{l+1}$ to one of $T$ or $B$ and $(i_l,i_{l+1})$ to one of $\mu_T,
  \mu_B$ or $\mu_{TB}$ according to the following table.
\begin{equation*}
\begin{array}{|c|c|c|c|c|}
\hline
i_l & \text{Term} & i_{l+1} & (i_l,i_{l+1}) & \text{Next }i_{l+1} \\
\hline
T & a & B & \mu_{TB} & T \\
T & b & T & \mu_T & B \\
B & a & T & \mu_{TB} & B \\
B & b & B & \mu_B & T \\
\hline
\end{array}
\end{equation*}
Increment $l$ by one.

\item Repeat Step~4 until you return to $i_1$, which will necessarily
  belong to $B$, since there are an even number of $b_{i,j}$'s in the
  term.

\item Increment $k$ by 1. 
\item Repeat Steps~3-6 until $k=m$, i.e., until all cycles are exhausted.
\end{enumerate}
\end{alg}

The following result is now an easy consequence.

\begin{lem} \label{lem:bij}
Algorithms~\ref{alg:pmtoperm} and \ref{alg:permtopm} are
inverses of each other.
\end{lem}

\section{The Crossing Number} \label{sec:cross}
Now that we have established a bijection between terms in the
determinant expansion and labeled Brauer diagrams, we need to show
that the sign associated to both of these are the same.  We start with
a labeled Brauer diagram $\mu$, which leads to a permutation 
$\pi=C_1\dots C_m$ and a product of $a$'s and $b$'s according to
Algorithm~\ref{alg:pmtoperm}. Let $\tau$ be the same product obtained
from the determinant expansion of the matrix using permutation $\pi$
{\em including the sign}.  From the definition of the matrix
\eqref{defm}, we will first write a formula for the sign associated to
$\tau$.

Let $C_j = (n^{(j)}_1,\dots,n^{(j)}_{l(j)})$. Then, define the sequences
$\beta^{(j)}$ (resp. $\gamma^{(j)}$) of length $l(j)$ consisting of terms $\pm 1$ (resp. $\pm i$) according to the following definition.
\be
\beta^{(j)}_i = \begin{cases}
+1 & n^{(j)}_i < n^{(j)}_{i+1},\\
-1 & n^{(j)}_i > n^{(j)}_{i+1},
\end{cases};
\quad
\gamma^{(j)}_i = \begin{cases}
+i & n^{(j)}_i < n^{(j)}_{i+1},\\
-i & n^{(j)}_i > n^{(j)}_{i+1},
\end{cases}
\ee
where $n^{(j)}_{l(j)+1} \equiv n^{(j)}_1$. Then the sign associated to
the term $\tau$ depends on whether $\mu$ belongs to $(B_{F})_{n}$ 
or $(B_{B})_{n}$. In the former case, we have the formula
\be \label{signtau1}
\sgn(\tau) = (-1)^{\inv(\pi)} \prod_{j=1}^m \;\; \prod_{\substack{i=1
    \\ \ds b_{\widehat{n^{(j)}_i,n^{(j)}_{i+1}}} \in \tau  }}^{l(j)}
\beta^{(j)}_i.
\ee
and in the latter,
\be \label{signtau2}
\sgn(\tau) = (-1)^{\inv(\pi)} \prod_{j=1}^m \;\; \prod_{\substack{i=1
    \\ \ds b_{\widehat{n^{(j)}_i,n^{(j)}_{i+1}}} \in \tau  }}^{l(j)}
\gamma^{(j)}_i.
\ee
Since the number of $b$'s in the second product
 is even for all $j$, the product in \eqref{signtau2}
will necessarily be real and equal to $\pm 1$. 

First we look at Brauer diagrams with no cups or caps. There are no
$b_{i,j}$'s in the associated term in the determinant expansion.

\begin{lem} \label{lem:onlyas}
Suppose $\mu$ is a labeled Brauer diagram such that $\mu_T = \mu_B = \emptyset$ and let $\pi$ be the associated permutation. Then, if $\mu \in (\B_{F})_{n}$, then
\be
\inv(\pi) = \chi(\mu),
\ee
and if $\mu \in (\B_{B})_{n}$, then
\be
\inv(\pi)+\chi(\mu) = \binom n2.
\ee
\end{lem}

\begin{proof}
The former is obvious since $\mu$ is identical to the two-line diagram for $\pi$. 
The latter requires just a little more work.
For a matching with only arcs, the edges are exactly given
by $(i,\pi_i)$ for $i \in [n]$.
Now consider two edges
$(i,\pi_i)$ and $(j,\pi_j)$ where $i<j$, without loss of
generality. Recall that $i,j \in T$ and $\pi_i,\pi_j \in B$ by
convention. Then $(i,\pi_i)$ intersects $(j,\pi_j)$ if and only if
$\pi_i< \pi_j$ because of the right-to-left numbering convention in $B$.
Thus,
\be
\chi(\mu) = |\{(i,j)| i<j, \, \pi_i <\pi_j \}|.
\ee
On the other hand, the definition of an inversion number is
\be
\inv(\pi) = |\{(i,j)| i<j, \, \pi_i >\pi_j \}|.
\ee
Since these two count disjoint cases, which span all possible pairs
$(i,j)$, they must sum up to the total number of possibilities $(i,j)$
where $i<j$, which is exactly $\binom n2$.
\end{proof}

Now we will see what happens to the crossing number of a matching when
a cup and a cup are converted to two arcs.

\begin{lem} \label{lem:chordba}
All other edges remaining the same, for any $i,j,k,l$, the following results hold.
\vspace{0.2cm}
\setlength{\unitlength}{1mm}

\begin{enumerate}
\item[(a)]
 \be
  \label{chordba}
(-1)^{\ds \chi \Bigg(
\begin{picture}(40, 5)
\put(0,-5){\line(1,0){40}}
\put(0,5){\line(1,0){40}}
\qbezier(5, -5)(20,-1)(35, -5)
\qbezier(10,5)(20, 2)(30, 5)
\put(5,-9){k}
\put(35,-9){l}
\put(10,7){i}
\put(30,7){j}
\end{picture}
\Bigg)}
=
(-1)^{\ds \chi \Bigg(
\begin{picture}(40, 5)
\put(0,-5){\line(1,0){40}}
\put(0,5){\line(1,0){40}}
\put(5,-5){\line(1,2){5}}
\put(35,-5){\line(-1,2){5}}
\put(5,-9){k}
\put(35,-9){l}
\put(10,7){i}
\put(30,7){j}
\end{picture}
\Bigg)}.\nonumber
\ee

\item[(b)] \vspace{.5cm}
\be \label{chordba2}
(-1)^{\ds \chi \Bigg(
\begin{picture}(40, 5)
\put(0,-5){\line(1,0){40}}
\put(0,5){\line(1,0){40}}
\qbezier(5, -5)(17.5,0)(30, 5)
\qbezier(10,5)(22.5, 0)(35, -5)
\put(5,-9){k}
\put(35,-9){l}
\put(10,7){i}
\put(30,7){j}
\end{picture}
\Bigg)}
=
-(-1)^{\ds \chi \Bigg(
\begin{picture}(40, 5)
\put(0,-5){\line(1,0){40}}
\put(0,5){\line(1,0){40}}
\put(5,-5){\line(1,2){5}}
\put(35,-5){\line(-1,2){5}}
\put(5,-9){k}
\put(35,-9){l}
\put(10,7){i}
\put(30,7){j}
\end{picture}
\Bigg)}. \nonumber
\ee

\item[(c)] \vspace{.5cm}
\be \label{chordba3}
(-1)^{\ds \chi \Bigg(
\begin{picture}(40, 5)
\put(0,-5){\line(1,0){40}}
\put(0,5){\line(1,0){40}}
\qbezier(5, 5)(20,1)(35, 5)
\qbezier(10,5)(20, 0)(30, -5)
\put(5,7){i}
\put(35,7){k}
\put(10,7){j}
\put(30,-9){l}
\end{picture}
\Bigg)}
=
-(-1)^{\ds \chi \Bigg(
\begin{picture}(40, 5)
\put(0,-5){\line(1,0){40}}
\put(0,5){\line(1,0){40}}
\qbezier(5, 5)(17.5,0)(30, -5)
\qbezier(10,5)(22.5, 1)(35, 5)
\put(5,7){i}
\put(35,7){k}
\put(10,7){j}
\put(30,-9){l}
\end{picture}
\Bigg)}. \nonumber
\ee

\item[(d)] \vspace{.5cm}
\be \label{chordbb}
(-1)^{\ds \chi \Bigg(
\begin{picture}(40, 5)
\put(0,-5){\line(1,0){40}}
\put(0,5){\line(1,0){40}}
\qbezier(5, 5)(20,-2)(35, 5)
\qbezier(10,5)(20, 2)(30, 5)
\put(5,7){i}
\put(35,7){l}
\put(10,7){j}
\put(30,7){k}
\end{picture}
\Bigg)}
=
-(-1)^{\ds \chi \Bigg(
\begin{picture}(40, 5)
\put(0,-5){\line(1,0){40}}
\put(0,5){\line(1,0){40}}
\qbezier(5, 5)(17.5, -1)(30, 5)
\qbezier(10,5)(22.5, -1)(35, 5)
\put(5,7){i}
\put(35,7){l}
\put(10,7){j}
\put(30,7){k}
\end{picture}
\Bigg)}. \nonumber
\ee

\end{enumerate}

\end{lem}
\vspace{0.2cm}
\begin{proof}
We will prove the result only for (a). The idea of the proof is identical for all other cases.
We consider all possible edges that could intersect with any of the 4
edges $(i,j), (k,l), (i,l)$ and $(j,k)$ illustrated above. We group
them according to their position. 
\begin{enumerate}
\item
Let $n_{ij}$ (resp. $n_{kl}$) be the number of edges such that
exactly one of its endpoints lies between $i$ and $j$ (resp. $k$ and
$l$), and the other endpoint does not lie between $k$ and $l$
(resp. $i$ and $j$). These edges intersect $(i,j)$ (resp. $(k,l)$)
and do not intersect $(k,l)$ (resp. $(i,j)$). They also intersect
exactly one among $(i,l)$ and $(j,k)$.

\item Let $n_{ijkl}$ be the number of edges one of whose endpoints
  lies between $i$ and $j$, and the other, between $k$ and $l$. These
  intersect both $(i,j)$ and $(k,l)$.

\item Let $n_{LR}$ be the number of edges, one of whose endpoints is
  less than $k$ if it belongs to the top row and more than $j$ in the
  bottom row, and the other is more than $l$ in the top row or less
  than $i$ in the bottom row. These are edges which do not intersect
  either $(i,j)$ or $(k,l)$, but intersect both $(i,l)$ and $(j,k)$.
\end{enumerate}

Now, the contribution of the edges $(i,j)$ and $(k,l)$ to $\chi$ in
the left hand side of \eqref{chordba} is $n_{i,j}+n_{kl}+2n_{ijkl}$,
whereas that to the right hand side of \eqref{chordba} is
$n_{ij}+n_{kl}+2n_{LR}$. Since all other edges are the same, the
difference between the crossing number of the configuration on the
left and that on the right is $2n_{ijkl}-2n_{LR}$ and hence, the parity
of both crossing numbers is the same.
\end{proof}

\section{The Main Result} \label{sec:main}
We now prove the theorem in a purely combinatorial way.  The proof
will depend on whether the Brauer diagram belongs to $(\B_{F})_{n}$ or
$(\B_{B})_{n}$, but the idea is very similar in both cases. We will
prove the former and point out the essential difference in the proof
of the latter at the very end.

\begin{proofof}{Theorem~\ref{thm:detm}}
From Lemma~\ref{lem:bij}, we have shown that every term in the
expansion of the determinant corresponds, in an invertible way, to a
Brauer diagram. We will now show the signs are also equal by performing
an induction on the number of cups, or equivalently caps, since both are the same.

Consider a $F$-Brauer diagram $\mu \in (\B_{F})_{n}$ 
with at least one cup and cap each. 
Using the bijection of Lemma~\ref{lem:bij}, construct the associated
permutation $\pi$. By the construction in
Algorithm~\ref{alg:pmtoperm}, there have to be at least two $b$'s in
the same cycle $C$, say. We pick two of them such that $(i,j) \in \mu_B$ is a 
cup and $(k,l) \in \mu_T$ is a cap. We have to show that $(-1)^{\chi(\mu)}= \sgn(\tau)$ using  \eqref{signtau2}.

We now get a new Brauer diagram $\mu' \in (\B_{F})_{n}$ by replacing the cup 
$(i,j)$ and the cap $(k,l)$  by the arcs $(i,k)$ and $(j,l)$ using 
Lemma~\ref{lem:chordba}(a). 
This replaces the associated weights $b_{i,j}b_{k,l}$ with $a_{\widehat{i,k}}
a_{\widehat{j,l}}$, and the sign remains the same, 
$(-1)^{\chi(\mu)} = (-1)^{\chi(\mu')}$. Now we use the same algorithm to
construct the permutation $\pi'$ associated to the new term, and look
at how the cycle $C$ changes to $C'$. Let $\tau$ and $\tau'$ be  terms
 obtained in the determinant expansion of $M_F$ including the sign.

There are four ways in which these 4 numbers are arranged in $C$. We
list these and the way they transform in Table~\ref{tab:cycdec1}. 
In each case, the links
$\{i,j\}$ and $\{k,l\}$ are broken and the links $\{i,k\}$ and
$\{j,l\}$ are formed. Recall that $i<j$ and $k<l$ according to 
Lemma~\ref{lem:chordba}(a).

\begin{table}[h!]
\begin{tabular}{|c|c|c|c|}
\hline
$C \in \pi$ & $C' \in \pi'$ & Factors in $\pi$ & 
Relative sign \\
\hline
$(i,j,\dots,k,l,\dots)$ &$(i,k,\dots,j,l,\dots)$  & $b_{i,j} b_{k,l}$ & $+1$ \\
$(i,j,\dots,l,k,\dots)$ &$(i,k,\dots)(j,\dots,l)$ & $b_{i,j} (-b_{k,l})$ & $-1$ \\
$(j,i,\dots,k,l,\dots)$ &$(j,l,\dots)(i,\dots,k)$ & $(-b_{i,j}) b_{k,l}$ & $-1$ \\
$(j,i,\dots,l,k,\dots)$ &$(j,l,\dots,i,k,\dots)$ & $(- b_{i,j})(-b_{k,l})$ & $+1$ \\
\hline
\end{tabular}
\vspace{0.2cm}
\caption{Comparison between the difference of the number of cycles in
  $C$ and $C'$, and the relative sign between the factor in $\pi$ and
  $a_{\widehat{i,k}}a_{\widehat{j,l}} \in \pi'$.} \label{tab:cycdec1}
\end{table}

We now need an result from undergraduate combinatorics.  
When $n$ is odd (resp. even), a permutation $\pi$ of size $n$ is odd
if and only if the number of cycles is even (resp. odd) in its cycle
decomposition. Therefore, the parity of the
permutation $\pi'$ is different from $\pi$ in cases (1) and (4) and
the same as that of $\pi$ in cases (2) and (3). Notice that the
relative signs also follow the same pattern.

To summarize, we have shown that $(-1)^{\chi(\mu)}= \sgn(\tau)$ holds
if and only if $(-1)^{\chi(\mu')}= \sgn(\tau')$ holds when $\mu,\mu' \in (\B_{F})_{n}$. 
But this is
precisely the induction step since $\mu'$ and $\mu''$ have one less cup and one less cap that $\mu$. 
From Lemma~\ref{lem:onlyas}, we have already shown that the
terms which correspond to Brauer diagrams with only arcs
have the correct sign. 
This completes the proof.

We follow the same strategy when $\mu$ belongs to $(\B_{B})_{n}$. 
The difference is that $l<k$ and that $b_{i,j}$ and $b_{k,l}$ come with
additional factors of $\imag$. The interested reader can check that these two 
contribute opposing signs leading to the same result.
\end{proofof}

For even antisymmetric matrices, this gives a natural combinatorial interpretation of Cayley's theorem different from the ones given by Halton \cite{halton} and E{\u{g}}ecio{\u{g}}lu \cite{egecioglu}. 

\begin{cor}[Cayley 1847, \cite{Cayley}]
For an antisymmetric matrix $M$ of size $n$, 
\be
\det M = \begin{cases}
(\pf M)^2 & \text{$n$ even}, \\
  0 & \text{$n$ odd}.
\end{cases}
\ee
\end{cor}

\begin{proof}
From \eqref{defm}, we see that all $a_{i,j}$'s are zero for
an antisymmetric matrix for both $M_{F}$ and $M_{B}$. We consider only the former representation since the argument is identical for the latter.
The only $F$-Brauer diagrams in $(\B_{F})_{n}$ that
contribute are those with no arcs. If $n$ is odd,
this is clearly not possible. Thus the determinant is zero.
If $n$ is even, we have the sum in Theorem~\ref{thm:detm} over all Brauer diagrams with only cups and cups.  This sum  now factors into two distinct sums for cups and for caps. But for each of these cases, we know that the answer is the same since they are independent sums. Moreover, each of these is the Pfaffian \cite{stem}.
\end{proof}

It would be interesting to find an analogous expression for the
permanent of a matrix. This might entail finding a different planar
graph instead of a Brauer diagram 
or a different analog of the crossing number or both. 
For example, the permanent of  the matrix  in \eqref{mateg3} is given by
\be \label{permeg3}
\begin{split}
&\text{Perm}(M^{(3)}_{F})=a_{1,3}^2 a_{2,2}+a_{2,3}^2 a_{1,1} +a_{1,2}^2
a_{3,3}- b_{1,2}^2 a_{3,3}-b_{1,3}^2 a_{2,2} -b_{2,3}^2 a_{1,1}  \\
&+2 a_{1,2} a_{1,3} a_{2,3}+a_{1,1} a_{2,2} a_{3,3}-2 a_{2,3} b_{1,2}b_{1,3}
-2 a_{1,2} b_{1,3}b_{2,3}+2 a_{1,3} b_{1,2} b_{2,3}.
\end{split}
\ee
Note that not all signs in the permanent expansion of are positive. 

\section*{Acknowledgements}
This work was motivated by discussions with Craig Tracy, whom we thank
for encouragement and support. We also thank Ira Gessel, David M. Jackson, Christian Krattenthaler, Greg Kuperberg,  Dan Romik, Alexander Soshnikov and Doron Zeilberger for constructive feedback. We also thank a referee for a very careful reading of the manuscript which led to many improvements.

\bibliographystyle{alpha}
\bibliography{hermit}

\end{document}